\documentclass[12pt]{amsart}

\usepackage{DKstyle}

\newcommand{\vertiii}[1]{{\left\vert\kern-0.25ex\left\vert\kern-0.25ex\left\vert #1
    \right\vert\kern-0.25ex\right\vert\kern-0.25ex\right\vert}}
\theoremstyle{plain}

\usepackage{caption}
\usepackage[labelfont=rm]{subcaption}

\usepackage[pagebackref=true, colorlinks]{hyperref}

\hypersetup{pdffitwindow=true,linkcolor=blue,citecolor=blue,urlcolor=blue,menucolor=blue}

\usepackage{comment}


\date{\today}%
\dedicatory{}%
\commby{}%

\title{Values of random polynomials in shrinking targets}
\author{Dubi Kelmer}
\address{Department of Mathematics, Boston College, Chestnut Hill MA 02467-3806, USA}
\email{kelmer@bc.edu}
\author{Shucheng Yu}
\address{Department of Mathematics, Technion, Haifa, Israel}
\email{yushucheng@campus.technion.ac.il}
\thanks{The authors were partially supported by NSF CAREER grant DMS-1651563. The second author was supported by ERC grant HD-APP}
\begin{document}
\begin{abstract}
Relying on the classical second moment formula of Rogers we give an effective asymptotic formula for the number of integer vectors $v$ in a ball of radius $t$, 
with value $Q(v)$ in a shrinking interval of size  $t^{-\kappa}$, that is valid for almost all indefinite quadratic forms in $n$ variables for any $\kappa<n-2$. This implies in particular, the existence of such integer solutions establishing the prediction made by Ghosh Gorodnik and Nevo \cite{GhoshGorodnikNevo2018}. We also obtain similar results for random polynomials of higher degree.
\end{abstract}
\maketitle

\section{Introduction}
Let $Q$ be a non-degenerate indefinite quadratic form in $n\geq 3$ variables. We say $Q$ is \textit{irrational} if $Q$ is not a multiple of a quadratic form with rational coefficients. The Oppenheim Conjecture, proved by Margulis \cite{Margulis1987}, states that if $Q$ is irrational, then $\overline{Q(\Z^n)}=\R$.
Going beyond this one can ask about an effective rate for the density, that is, given $\xi\in \R$ and a large parameter $t>0$, one would like to establish how small can 
$|Q(v)-\xi |$ be, for $v\in \Z^n$ with $\|v\|\leq t$ bounded. This type of problem has a long history \cite{BirchDavenport1958, BentkusGotze1999, GotzeMargulis2010}, and we refer to \cite{Margulisfields1997} for an extensive review. As an example we note that for $n\geq 5$ it was shown in \cite{GotzeMargulis2010} , that under a suitable diophantine condition on the coefficients of $Q$ there is $\kappa>0$ such that the inequality
$$|Q(v)-\xi|<\|v\|^{-\kappa}$$
has infinitely many integer solutions (when $\xi=0$ this holds for all forms). For ternary forms the best bounds are due to Lindenstrauss and Margulis \cite{LindenstraussMargulis2014} who showed that under suitable diophantine conditions on $Q$ the inequality $|Q(v)-\xi|\leq \log(\|v\|)^{-\kappa}$ has infinitely many integer solutions. 

Improving on these bounds for any given form $Q$ seems like a very difficult problem, nevertheless, much more can be said when considering a generic form.
In \cite{GhoshGorodnikNevo2018}, Ghosh Gorodnik and Nevo considered the problem of values of generic polynomials and gave a heuristic argument based on the pigeon hole principal predicting that for a generic degree $d$ polynomial  $F$ in $n$ variables, one should expect that the system of inequalities
\begin{equation} \label{e:main} |F(v)-\xi|<t^{-\kappa},\; \|v\|\leq t\end{equation}
would have integer solutions for any positive $\kappa<n-d$, and in particular, for quadratic forms this should hold for any $\kappa<n-2$.

To make the notion of a generic form more precise, we denote by $Y_{p,q}$ the space of determinant one quadratic forms of signature $(p,q)$ and note that the natural action of $\SL_n(\R)$ on this space (via change of variables) is transitive, and hence the Haar measure  of $\SL_n(\R)$ gives a natural measure on $Y_{p,q}$.
In this setting, by  utilizing the fact that an indefinite quadratic form is stabilized by a large semisimple group, 
and studying a shrinking target problem for the action of this group, \cite{GhoshGorodnikNevo2018} showed that there is some $\kappa_0$ such that for all $\kappa<\kappa_0$ for any $\xi\in \R$, for almost all forms $Q\in Y_{p,q}$ the inequality \eqref{e:main} has integer solutions for all sufficiently large $t$. While in general the value of $\kappa_0$ is smaller than $n-2$,  for $n=3$ they show that $\kappa_0=1$ in agreement with the heuristic prediction (see also \cite{GhoshKelmer17} for a similar result for $\xi=0$). For the special case where $\xi=0$, in \cite{AthreyaMargulis2018}  Athreya and Margulis used a completely different approach relying on lattice point counting, and showed that for any $n\geq 3$ and for any $\kappa<n-2$, for almost all $Q\in Y_{p,q}$ there are integer solutions to 
$|Q(v)|<t^{-\kappa}$ with $\|v\|\leq t$ for all sufficiently large $t$.

A different way to try and quantify the density of integer values of forms, is to study the asymptotics for the number of integer solutions $v\in \Z^n,\; \|v\|\leq t$ with
$Q(v)\in I$ for some fixed small interval $I$. Here Eskin, Margulis and Mozes \cite{EskinMargulisMozes1998} showed that for any irrational quadratic form $Q$ of signature $(p,q)$ with $p\geq 3$, $q\geq 1$, and any interval $I\subseteq \R$ the number of solutions is asymptotic to $c_Q |I|t^{n-2}$ with $c_Q$ an explicit constant depending on the form $Q$. This is no longer true for forms of signature $(2,2)$ or $(2,1)$ where one can find examples for which the number of solutions grows logarithmically faster than $c_Q|I|t^{n-2}$. Nevertheless, they showed that the same asymptotic holds for almost all quadratic forms of signature $(2,2)$ or $(2,1)$. 

As in the problem for the rate of density, for  this problem one can also expect more when considering a generic form. Indeed, \cite{AthreyaMargulis2018} improved the asymptotic formula to give an effective estimate with a power saving. Explicitly, they showed that there is $\nu>0$ such that for any fixed interval $I$ and for almost all $Q\in Y_{p,q}$, 
$$\#\{v\in \Z^n:  Q(v)\in I,\; \|v\| \leq t\}=c_Q|I|t^{n-2}+O_{Q,I}(t^{n-2-\nu}),$$
where here and below we use notation $A=O(B)$ to mean that $A\leq cB$ for some constant $c>0$, and we use the subscript to emphasize the dependance of this constant on additional parameters.

In this paper, we refine the result of \cite{AthreyaMargulis2018}, by considering the same problem when we allow the interval $I$ to shrink as $t$ grows. As mentioned in \cite{AthreyaMargulis2018}, this method is suitable to deal with polynomials of higher degree, and we illustrate this by considering the problem in this generality.  To do this, for $d\geq 2$ and $n=p+q$ let
$$F_0(v)=\sum_{i=1}^pv_i^d-\sum_{i=p+1}^n v_i^d,$$ 
and consider the space $Y_{p,q}^{(d)}$ of homogenous polynomials of degree $d$ that are of the form $g\cdot F_0$ with $g\cdot F_0(v)=F_0(vg)$ and $g\in \SL_n(\R)$.
The Haar measure of $\SL_n(\R)$ then gives a measure on $Y_{p,q}^{(d)}$, giving us a natural notion of almost all polynomials in this space.
We note that when $d=2$ the space $Y_{p,q}^{(2)}=Y_{p,q}$ is the full space of determinant one quadratic forms of signature $(p,q)$.  

\begin{thm}\label{thm:singleapprox}\label{t1}
For any $d\geq 2$ even and $n=p+q>d$ with $p,q\geq 1$, let $0<\kappa<n-d$. Let $\{I_t\}_{t> 0}$ be a decreasing family of bounded measurable subsets of $\R$ with measures $|I_t|=ct^{-\kappa}$ for some $c>0$. Then there is $\nu>0$ such that for almost all $F\in Y^{(d)}_{p,q}$ there is a constant $c_F>0$ such that 
$$\#\{v\in \Z^n: F(v)\in I_t,\; \|v\|\leq t\}=c_F |I_t|t^{n-d}+O_F(t^{n-d-\kappa-\nu}).$$

\end{thm}

\begin{rem}
From our proof one can extract an explicit value for $\nu$, 
and in particular any $\nu<\frac{2(n-d-\kappa)}{n^2+n+4}$ will work. We did not try to obtain the optimal power saving here and our main point is that there is some positive power saving.
\end{rem}

This result is valid for any family of shrinking targets, in particular, taking the shrinking sets to be the intervals $I_t=(\xi-t^{-\kappa},\xi+t^{-\kappa})$ it implies the following corollary, verifying the prediction of \cite{GhoshGorodnikNevo2018}.
\begin{cor}
Let $n=p+q>d$ be as in Theorem \ref{thm:singleapprox}. For any $0<\kappa<n-d$ and for any $\xi\in \R$, for almost all $F\in Y^{(d)}_{p,q}$ the system of inequalities 
$$|F(v)-\xi|<t^{-\kappa},\quad \|v\|\leq t,$$
has integer solutions for all sufficiently large $t$.
\end{cor}

For the results described above, one first fixes the shrinking sets, or for the case of intervals the center point $\xi$, and only then obtain a result for almost all polynomials, so that this full measure set of polynomials may depend on $\xi$.  A natural question is then, how well can one polynomial (chosen at random), approximate all target points $\xi$?  This question was addressed in \cite{Bourgain2016} for the case of indefinite diagonal ternary quadratic forms, and in \cite{GhoshKelmer2018} for general indefinite ternary quadratic forms. In these cases they showed that given a sequence $N(t)$ and $\delta(t)$ such that $\frac{N(t)}{t^{a}\delta(t)^2}\to 0$ with $a<1$ then for almost all $Q \in Y_{2,1}$ and for all sufficiently large $t$,
$$\sup_{|\xi|\leq N(t)}\min_{v\in \Z^n,\;\|v\|\leq t}|Q(v)-\xi|<\delta(t).$$

Using our method we are also able to give the following effective counting estimate in this setting. 
\begin{thm}\label{thm:sim2}\label{t2}
Let $n=p+q>d$ be as in Theorem \ref{thm:singleapprox}. Let $0\leq \eta<\min\{d, n-d\}$ and let $0< \kappa< \tfrac{n-d-\eta}{2}$. Let $N(t)$ be a non-decreasing function satisfying that $N(t)=O(t^\eta)$. Then there is $\nu>0$ such that for almost all $F\in Y_{p,q}^{(d)}$  for any interval $I\subset [-N(t),N(t)]$ with $|I|\geq t^{-\kappa}$ we have
\begin{equation}\label{equ:allint}
\#\{v\in \Z^n: F(v)\in I,\; \|v\|\leq t\}=c_F|I|t^{n-d}+O_F(|I|t^{n-d-\nu}).
\end{equation}
\end{thm}

\begin{rem}
Taking the function $N(t)$ to be a constant implies that in order for \eqref{e:main} to have integer solutions for all sufficiently large $t$ and for all $\xi$ in some compact set, we need an exponent
$\kappa<\tfrac{n-d}{2}$ rather than $n-d$ as we got for a fixed $\xi$. It is unclear if this is really the best one can hope for or if it is just an artifact of the proof.
\end{rem}
As a consequence we get the following generalization of the result of \cite{GhoshKelmer2018} to higher dimensions $n\geq 3$ as well as higher degrees.
\begin{cor}
Let $n=p+q>d$ be as in Theorem \ref{thm:singleapprox}. Given a non-decreasing functions $N(t)=O(t^\eta)$  with $\eta<\min\{d,n-d\}$ and a non-increasing function $\delta(t)$ satisfying that $\frac{t^{\eta-a}}{\delta(t)^2}\to 0$ for some $a<n-d$,  
we have that for almost all  $F\in Y^{(d)}_{p,q}$ and for all sufficiently large $t$
$$\sup_{|\xi|\leq N(t)}\min_{v\in \Z^n,\; \|v\|\leq t}|F(v)-\xi|<\delta(t).$$
\end{cor}

%{\color{red}{Maybe we can restate this Corollary for $N(t)$ with $N(t)=O(t^{\eta})$ and replace the condition $\frac{N(t)}{t^a\delta(t)^2}\to 0$ by the condition $\frac{t^{\eta-a}}{\delta(t)^2}\to 0$ (as for this we will have slightly more freedom for $N(t)$) and then make a remark that when $N(t)\asymp t^{\eta}$ this is equivalent to $\frac{N(t)}{t^a\delta(t)^2}\to 0$.}}
%\begin{rem}
%Note that taking $\kappa=\frac{n-2-\eta-\e}{2}$ our assumptions on $N(t)$ and $I$ imply that $\frac{N(t)}{|I|^2}\ll t^{\eta+2\kappa}=t^{n-2-\e}$,
%In particular, for $n=3$ we recover the result of \cite{GhoshKelmer17}, and also give an estimate of the number of solutions..
%\end{rem}

\subsection{Outline of proof}
As some parts of the proof can get a bit technical, for readers' convenience we outline here the general strategy. The main idea is the following general principle: Given a nice enough large set in $\R^n$ we expect the number of lattice points in the set to be close to its volume. In particular, we consider here sets of the form 
$$F^{-1}(I)\cap B_t=\{v\in \R^n: F(v)\in I,\; \|v\|\leq t\},$$
whose volume is expected to grow like  $c_F|I|t^{n-d}$. Here and below we denote by $B_t\subset \R^n$ the closed ball centered at the origin with radius $t$. In particular, for $|I|$ of order $t^{-\kappa}$ with $\kappa<n-d$ the volume of these sets grows with $t$ and we expect them to contain integer points. 

More explicitly, writing an element  $F\in Y_{p,q}^{(d)}$ as $F(v)=F_0(vg)$ with $g\in \SL_n(\R)$, since $F^{-1}(I)=F_0^{-1}(I)g^{-1}$ we can write 
$$\#\{v\in \Z^n: \|v\|\leq t,\; F(v)\in I\}=\#(\Z^ng\cap F_0^{-1}(I)\cap B_tg).$$
Next we recall the result of Schmidt \cite{Schmidt1960}, relying on Rogers' second moment formula \cite{Rogers1955}, who showed that given any increasing family of sets $A_t$ in $\R^n$ for almost all lattices $\Lambda=\Z^n g$ with $g\in \SL_n(\R)$ one has that $\#(\Lambda \cap A_t)=\vol(A_t)+O(\sqrt{\vol(A_t)}\log^2(\vol(A_t))$. When the interval $I$ is fixed,
the family $A_t=F_0^{-1}(I)\cap B_t$ is an increasing family and hence we have a very good estimate for  $\#(\Lambda\cap A_t)=\#(\Z^ng\cap F_0^{-1}(I)\cap B_t)$.
 This is still not enough, since for our purpose, the expanding sets $F_0^{-1}(I)\cap B_tg$ also depend on $g$.  To overcome this problem we replace them with sets of the form $F_0^{-1}(I)\cap B_th$ with $h\in \SL_n(\R)$ fixed and taken from a sufficiently dense set.

This was the approach used in \cite{AthreyaMargulis2018} for the case of a fixed interval. When considering shrinking intervals $I_t$, we encounter another difficulty, that our family of sets $F_0^{-1}(I_t)\cap B_tg$ is no longer an increasing family so the results of Schmidt do not apply. The main new ingredient in our proof is using a different (simpler) interpolation argument, relying again on Rogers' second moment formula, which allows us to handle families of sets that are increasing in one aspect and decreasing in another, as long as their volume grows sufficiently regularly. Of course, we pay a price that we no longer have a square-root bound for the remainder, but we do get some power saving which is sufficient for our result. 

We further note that, for this approach to work, it is not enough to know the asymptotics of $\vol(F^{-1}(I)\cap B_t)$ as $t\to\infty$, and one needs an explicit estimate for the volume with a power saving  bound for the remainder of the form
\begin{equation}\label{e:vol}
 \vol(F^{-1}(I)\cap B_t)=c_F|I|t^{n-d}+O_F(|I|t^{n-d-\nu}).
\end{equation}
Following some preliminary results in section \ref{sec:setup}, we devote section \ref{s:vol} to establish such a volume estimate. Then in section \ref{sec:single} we use this approach to prove Theorem \ref{t1}. Then in section \ref{sec:uniform} we follow a similar argument, but instead of considering a single target $\xi$ we consider a sufficiently dense collection of targets at once, in order to get the uniform estimate in Theorem \ref{t2}.
 
 \begin{rem}
 We note that this method is quite soft and works in general as long as one has a volume estimate of the form \eqref{e:vol}. While we establish this estimate here only for homogenous polynomials in the orbit $Y_{p,q}^{(d)}$ for even $d$, we expect that such an estimate (and hence similar results on integer values) should hold also for other homogenous polynomials. 
 \end{rem}
 
 \begin{rem}
Another key ingredient for this method is Rogers' second moment formula (see section \ref{s:disc}).  Recently, in \cite{KelmerYu2018}, we showed how such a second moment formula can be generalized from the space of lattices $\cX=\SL_n(\Z)\bk \SL_n(\R)$ to other homogenous spaces. Using such a generalization will allow one to apply this method in even greater generality. For example, replacing $\SL_n(\Z)$ with a congruence subgroup allows one to deal with certain inhomogeneous forms (see \cite{GhoshKelmerYu19}), and considering different semisimple groups $G$ allows one to tackle this problem for polynomials on other varieties with a transitive $G$-action as will be shown in \cite{KelmerYu19}.
 \end{rem}

 \section{Setup and some preliminary results}\label{sec:setup}
In this section, we set up some notation and provide some preliminary results that are needed for the proofs of our main theorems. 

\subsection{Notation}
In what follows we fix $n\geq 3$, $p, q\geq 1$ and $d\geq 2$ even with $n=p+q$ and $n>d$. 
Let $G=\SL_n(\R)$ and $\G=\SL_n(\Z)$ and $K=\SO(n)$. We denote by $\mu$ the Haar measure of $G$ normalized to be a probability measure on $\G\bk G$,

 For any $g\in G$ we denote by $\|g\|_{\rm op}$ the operator norm given by
$$\|g\|_{\rm op}=\sup\{\|vg\|: v\in \R^n,\; \|v\|=1\},$$
with $\|v\|=\sqrt{\sum_i v_i^2}$ the standard Euclidean norm on $\R^n$. We then define a symmetric norm on $G$ by
$$\|g\|=\max\{\|g\|_{\rm op}, \|g^{-1}\|_{\rm op}\}.$$

We will use the notation $A=O(B)$ as well as $A\ll B$ to indicate that there is a constant $c>0$ such that $A\leq c B$, and we will use subscripts to emphesize the dependence of this constant on additional parameters. We will also use the notation $A\asymp B$ to mean that $A\ll B\ll A$. Since we fix $p,q, n$ and $d$, all implied constants may depend on them.

\subsection{A covering lemma}
For $\e>0$ small we consider the norm balls 
$$\cO_\e=\{g\in G: \|g\|< 1+\e\},$$ 
and note that for any $g\in \cO_{\epsilon}$ with $0<\epsilon<1$ and any $t>0$ we have %$B_{\frac{t}{1+\epsilon}}\subset B_th\subset B_{(1+\epsilon)t}$. Moreover, since $1-\epsilon<\frac{1}{1+\epsilon}$ we have
\begin{equation}\label{equ:covering}
B_{(1-\epsilon)t}\subset B_t g\subset B_{(1+\epsilon)t}.
\end{equation}

Noting that the norm $\|g\|$ is right $K$-invariant, since the Haar measure is absolutely continuous with respect to the volume measure on the Lie algebra, we have that for any $0<\epsilon\leq1$
\begin{equation}\label{e:volO}
\mu(\cO_\e)\asymp  \e^{d_n},\end{equation}
with  $d_n=\textrm{dim}_{\R}(G/K)=\frac{(n+2)(n-1)}{2}$. Any compact set in $G$ can be covered by finitely many translates of $\cO_\e$, and we will use the following estimate for the number of translates that are needed.

\begin{Lem}\label{prop:covering}
Let $\cK\subset G$ be a fixed compact set. Then for any $0<\e\leq1$ there exists a finite set $\cI_{\epsilon}\subset \cK$ such that $\#\cI_{\epsilon}=O_{\cK}(\epsilon^{-d_n})$ and $\cK\subset \bigcup_{h\in \cI_{\epsilon}}\cO_{\epsilon}h$.
\end{Lem}
\begin{proof}
Since $\cK$ can be covered by a finite number of translates of $\cO_1$ it is enough to show this for $\cK=\cO_1$. Now for $\e>0$ let $\cI_\e$ be a maximal set of points in $\cO_1$ such that the translates $\cO_{\e/3}h$ with $h\in \cI_\e$ are pairwise disjoint. Note that for $0<\e\leq 1$ and for $h, h'\in \cO_1$, if $h'\not\in \cO_\e h$ then $\cO_{\e/3}h\cap \cO_{\e/3}h'=\emptyset$, and hence, the maximality of $\cI_\e$ implies that $\cO_1\subseteq \bigcup_{h\in \cI_\e} \cO_\e h$. Moreover, since $\bigcup_{h\in\cI_{\e}}\cO_{\e/3}h$ is a disjoint union contained in $\cO_2$ then  
$\mu(\cO_2)\geq \mu(\bigcup_{h\in\cI_{\e}}\cO_{\e/3}h)=\#\cI_\e \mu(\cO_{\e/3})$ so $\#\cI_\e\leq  \mu(\cO_2)\mu(\cO_{\e/3})^{-1}\ll \e^{-d_n}$ by \eqref{e:volO}. 
%Thus $1\ll \mu(\bigcup_{x\in \cI_\e} \cO_\e x)\asymp \#I_{\e}\epsilon^{d_n}$ implying that $\#I_{\epsilon}\gg \e^{-d_n}$. On the other hand 
%since for $0<\e<1$, $\cO_{\e/3}\cO_1\subset \cO_2$, we have $\bigcup_{x\in\cI_{\e}}\cO_{\e/3}x\subseteq \cO_2$ implying that $\#\cI_\e\ll \e^{-d_n}$. %concluding the proof.
\end{proof}

\subsection{Discrepancy bounds}\label{s:disc}
As explained in the introduction, we can translate the problem of counting integer values of a homogeneous polynomial to a problem of counting lattice points in a region of $\R^n$, and the notion of a generic polynomial can be translated to counting lattice points from a generic lattice.

Recalling our expectation that the number of lattice points in a set should be roughly the volume we define the discrepancy for a lattice $\Lambda$ in $\R^n$ and a finite-volume set $A\subseteq \R^n$ as 
\begin{equation}\label{e:disc}
D(\Lambda,A)=\left|\#(\Lambda \cap A)-\vol(A)\right|.
\end{equation}

By using Siegel's mean value formula together with Rogers' second moment formula we can get very good mean square bounds for the discrepancy when averaged over the space of lattices. Recall the space of rank $n$ unimodular lattices can be parametrized by $\cX=\G\bk G$, where we identify the coset $\G g$ with the  lattice $\Lambda=\Z^ng$, and let $\mu$ be the probability measure on $\cX$ coming from the Haar measure of $G$ as before.
We recall that the \textit{Siegel transform}, $\hat{f}:\cX\to \C$, of a bounded compactly supported function $f:\R^n\to \C$ is defined by
$$\hat{f}(\Lambda):=\sum_{ v\in \Lambda\bk \{0\}}f(v),$$
and the Siegel's mean value formula states that 
$$\int_{\cX}\hat{f}(\Lambda)d\mu(\Lambda)=\int_{\R^n}f(v)dv.$$
Moreover, a direct consequence of Rogers' second moment formula \cite{Rogers1955} for $n\geq 3$ implies that 
$$\int_{\cX}\left|\hat{f}(\Lambda)\right|^2d\mu(\Lambda)=\left|\int_{\R^n}f(v)dv\right|^2 +O\left(\int_{\R^n} \left|f(v)\right|^2dv\right).$$
In particular applying this estimate for $f=\chi_A$ the indicator function of a bounded measurable set $A\subseteq \R^n$ and noting that $\hat{f}(\Lambda)=\#((\Lambda\bk\{0\})\cap A)$ we get that
there is some $C_n'>0$ such that 
$$\int_{\cX} \left|D(\Lambda, A)\right|^2d\mu(\Lambda)\leq C'_n \vol(A)+1,$$
%\begin{Lem}[\cite{Rogers1956,AthreyaMargulis09}]\label{lem:rogers}
%For $n\geq 3$ there exists a constant $C_n$ only depending on $n$ such that for any bounded and finite-volume set $A\subset \R^n$
%$$\int_{\cX_n}\hat{\chi_A}^2d\mu_n\leq |A|^2+C_n|A|,$$
%where $\chi_A$ is the indicator function of $A$. In particular, together with Siegel's mean value theorem \cite{Siegel1945} this implies that 
%$$\int_{\cX_n}D(\Z^ng, A)^2d\mu_n(g)\leq C_n|A|.$$
%\end{Lem}
where the term $1$ is only needed if $A$ contains the origin. In particular, there exists some $C_n$ such that for all bounded measurable sets $A\subset \R^n$ with $\vol(A)>1$
\begin{equation}\label{e:DiscVar}
\int_{\cX} \left|D(\Lambda, A)\right|^2d\mu(\Lambda)\leq C_n \vol(A).
\end{equation}
Using this bound we get a good estimate on the measure of the collection of lattices with large discrepancy.
For a fixed compact set $\cK\subset G$ for any large parameter $T$ and a set $A\subseteq \R^n$ define the set 
\begin{equation}\label{e:MAT}
\cM^{(\cK)}_{A,T}:=\{g\in \cK:\;   D(\Z^ng,A)\geq T\}.
\end{equation}
We then get the following estimate.
 
\begin{Lem}\label{l:MAT}
Fix $\cK\subset G$ compact. For any bounded and measurable subset $A\subset \R^n$ with $\vol(A)>1$ and any $T>0$ 
%let $
%\cM_{A,T}\subset \cK$ be defined as
%$$\cM_{A,T}:=\{g\in \cK\ |\ D(\Z^ng,A)\geq T\}.$$ 
%Then 
$$\mu(\cM^{(\cK)}_{A,T})\ll_\cK \frac{\vol(A)}{T^2}.$$
%\begin{equation}\label{equ:rogers}
%\int_{\cK}D(\Z^ng, A)^2d\mu(g)\leq C_nL|A|,
%\end{equation}
\end{Lem}
\begin{proof}
Let $\cF\subseteq G$ be some fixed fundamental domain for $\cX$.
Since $G$ is tessellated by translates $\g\cF$ with $\g\in \G$ and $\cK$ is compact, it is covered by a finite number of translates, say, $\cF_L=\bigcup_{i=1}^L \g_i\cF$. Thus $\cM^{(\cK)}_{A,T}\subset \cK\subset \cF_L$ and we can bound
$$\mu(\cM^{(\cK)}_{A,T})\leq \frac{1}{T^2}\int_{\cM^{(\cK)}_{A,T}}|D(\Z^ng, A)|^2 d\mu(g)\leq \frac{1}{T^2}\int_{\cF_L}|D(\Z^ng, A)|^2d\mu(g).$$
Now, since any of the translates of $\cF$ is also a fundamental domain,  by \eqref{e:DiscVar} we can bound 
\begin{displaymath}
\int_{\cF_L}|D(\Z^ng, A)|^2d\mu(g)=L\int_{\cX}|D(\Z^ng, A)|^2d\mu(g)\leq C_nL\vol(A).\qedhere
\end{displaymath}
\end{proof}

We conclude this section with a simple interpolation argument relating the discrepancy of different sets, we omit the proof which is straightforward.
\begin{Lem}\label{l:inter}
For any finite-volume sets $A_1\subseteq A\subseteq A_2\subset\R^n$ and any $\Lambda\in \cX$ we have that
$$D(\Lambda,A)\leq \max\{D(\Lambda, A_1),D(\Lambda,A_2)\}+\vol(A_2\setminus A_1).$$
\end{Lem}
%\begin{proof}
%On one hand
%$$\#(\Lambda\cap A)-\vol(A)\leq \#(\Lambda\cap A_2)-\vol(A_1)\leq \#(\Lambda\cap A_2)-\vol(A_2)+\vol(A_2\setminus A_1)\leq D(\Lambda,A_2)+\vol(A_2\setminus A_1),$$
%and on the other hand
%$$\vol(A)-\#(\Lambda\cap A)\leq \vol(A_2)-\#(\Lambda\cap A_1)=\vol(A_2\setminus A_1)+\vol(A_1)-\#(\Lambda\cap A_1)\leq  D(\Lambda,A_1)+\vol(A_2\setminus A_1).$$
%\end{proof}

%To show the existence of such a full measure subset, we will show that its intersection with an arbitrary compact subset $\cK\subset \SL_n(\R)$ is of full measure in $\cK$. %fixing an arbitrary compact subset $\cK\subset \SL_n(\R)$ and showing that there exists a full measure subset  in $\cK$ such that for any $Q=g\cdot Q_0$ with $g$ in this full measure subset, the bound \eqref{equ:counting1} (resp. $\eqref{equ:allint}$) is satisfied. 
%
%For later use, for any $g\in \SL_n(\R)$ we denote by 
%$$\|g\|_{\textrm{op}}:=\sup\{\|vg\|\ |\ v\in \R^n\ \textrm{with}\ \|v\|=1\}$$
%the operator norm of $g$ when viewed as a linear operator on $\R^n$, where $\|v\|=\left(\sum_{i=1}^nv_i^2\right)^{\frac12}$ is the standard Euclidean norm of $v$. With slightly abuse of notation, we denote by $\|g\|:=\max(\|g\|_{\textrm{op}},\|g^{-1}\|_{\textrm{op}})$. For any compact subset $\cK\subset \SL_n(\R)$ we let
%\begin{equation}\label{equ:lambda}
%\lambda_{\cK}:=\max\{\|g\|\ |\ g\in \cK\}.
%\end{equation}
%Note that since $\cK$ is compact, $\lambda_{\cK}$ is finite.

\section{Volume estimates}\label{s:vol}
Given a homogenous polynomial $F$ of degree $d$ and an interval $I\subseteq \R$, the heuristic argument given in \cite{AthreyaMargulis2018} leads to the expectation that 
$\vol(F^{-1}(I)\cap B_t))$ is asymptotic to $c_F|I|T^{n-d}$. Such asymptotics were established for quadratic forms in \cite{EskinMargulisMozes1998}, and we refine their method and use it to give an explicit estimate with a power saving on the remainder for $\vol(F^{-1}(I)\cap B_t))$ for any $F\in Y_{p,q}^{(d)}$ with $d<p+q$ even.
%Let $d\geq 2$ be an even integer and consider the degree $d$ homogenous polynomial in $n=p+q$ variables of the form $F(v)=F_0(vg)$ with $g\in \SL_n(\R)$ and 
%$$F_0(v)=\sum_{i=1}^p v_i^d-\sum_{i=p+1}^n v_i^d.$$
%For such a polynomial we show the following.
\begin{thm}\label{t:vol}
Fix $d\geq 2$ even, and let $n=p+q> d$ with $p, q\geq 1$, and let $N\geq 1$. For $F\in Y_{p,q}^{(d)}$ and $I\subseteq [-N,N]$ measurable, there exists $c_F>0$ such that  for any $T> 2N^{1/d}$ we have 
\begin{equation}\label{equ:vol}
\vol(F^{-1}(I)\cap B_T)=c_F |I|T^{n-d}+O_F(|I| N^{1/d}T^{n-d-1}\log(T)),
\end{equation}
where the implied constant is uniform over compact sets and the $\log(T)$ factor is only needed when $n=2d-1$.
\end{thm}
%\begin{rmk}\label{rmk:unimodular}
%Note that for $F(v)=F_0(vg)$ with $g\in G$, since $F^{-1}(I)=F^{-1}_0(I)g^{-1}$ and $g$ is unimodular, we have $\vol(F^{-1}(I)\cap B_T)=\vol(F_0^{-1}(I)\cap B_Tg)$.
%\end{rmk}
%Using the facts that $g\in \SL_n(\R)$ is unimodular and $n-d-1\geq \max\{\frac{2n-3d}{2}, \frac{n-d-1}{2}\}$ for $d\geq 2$, we have the following immediate corollary that we will use for the proofs.
%\begin{cor}\label{cor:vol}
%With the same assumptions as above
%\begin{equation}\label{equ:vol}
%\vol(F_0^{-1}(I)\cap B_Tg)=c_g |I|T^{n-d}+O_g(|I| N^{1/d}T^{n-d-1}\log(T)),
%\end{equation}
%where the $\log(T)$ factor is needed only when $n=3$ and $d=2$.
%\end{cor}

%\begin{thm}\label{t:vol}
%Let $n=p+q\geq 3$ be as before and $N>0$ be a positive number. Let $Q(v)=Q_0(vg)$ with $g\in \SL_n(\R)$ and $I\subseteq [-N,N]$ a measurable set.. 
%Then for any $T> 3\sqrt{N}$ we have 
%\begin{equation}\label{equ:vol}
%\vol(Q^{-1}(I)\cap B_T)=c_Q |I|T^{n-2}+O_{p,q}(|I|\|g\|^{n-2} \sqrt{N}T^{n-3}\log(\|g\|T/\sqrt{N})),
%\end{equation}
%where the $\log(\|g\|T/\sqrt{N})$ factor is only needed when $n=3$.
%\end{thm}
We first give the following smoothed version.
\begin{Lem}\label{l:smoothvol}
Let $n=p+q> d$ be as in Theorem \ref{t:vol}. Let $h\in \cC_c^{\infty}(\R^n)$ be a nonnegative smooth function on $\R^n$ that is supported on $B_a$ for some $a>1$.
%the $\ell^d$-norm ball centered at the origin with radius $a$ 
For any measurable set $I\subset [-1,1]$ with indicator function $\chi_I$ and for any $T>1$ we have
\begin{equation}
\int_{\R^n}h(\frac{v}{T})\chi_I(F_0(v))dv = J(h) |I|T^{n-d}+\left\{\begin{array}{ll}
O\left(\cS_1(h)|I|a [aT]^{n-2d}\right)  & n\geq 2d+1\\
O\left(\|h\|_\infty  |I| \log(aT)\right)+O\left(\cS_1(h)|I|a\right) & n=2d\\
O\left(\|h\|_\infty  |I| \right)+O\left(\cS_1(h) |I| \frac{\log(aT)}{T}\right)   & n=2d-1\\
O\left(\|h\|_\infty  |I| \right)+ O\left( \frac{\cS_1(h)|I|}{T}\right)& n<2d-1,
 \end{array}\right.
\end{equation}
where
$$J(h)=\frac{1}{d}\int_{0}^{\infty}\int_{S_d^{p-1}\times S_d^{q-1}} h\big((\omega_1+\omega_2)r\big)r^{n-d-1}d\omega_1d\omega_2 dr,$$ 
where $S_d^{p-1}\subseteq \R^p$ and $S_d^{q-1}\subseteq \R^q$ denote the unit spheres with respect to the $\ell^d$-norm with $d\omega_1$, $d\omega_2$ the corresponding cone measures, and 
$$\cS_1(h)=\max\left\{\|h\|_\infty, \|\tfrac{\partial h}{\partial v_i}\|_\infty\;, 1\leq i\leq n\right\}.$$
\end{Lem}
%\begin{rmk}\label{rmk:inequ}
%Before giving the proof of this lemma, we record here an estimate that we will use extensively for future proofs: for any $0\leq x\leq 1$ and any $\alpha\in \R$, $(1+x)^{\alpha}=1+O_{\alpha}(x)$. 
%\end{rmk}
\begin{proof}
Denote by
$$\cI_{h,I}:=\int_{\R^n}h(\frac{v}{T})\chi_I(F_0(v))dv$$
the integral we want to compute. First we note that $\cI_{h,I}=\cI_{h,I\cap [-1,0)}+\cI_{h,I\cap [0,1]}$, so, up to replacing $F_0$ by $-F_0$, or equivalently, switching $p$ and $q$, we can assume that $I\subset [0,1]$.
Identify $\R^n=\R^p\times\R^q$ and write $v=(u_1,u_2)$ so that $F_0(v)=\|u_1\|_d^d-\|u_2\|_d^d$. For each factor we use spherical coordinates writing $u_i=r_i \omega_i$ 
with $r_i=\|u_i\|_d$ and 
$\omega_i=\frac{u_i}{\|u_i\|_d}$ in the unit sphere with respect to $\ell^d$-norm. To simplify notation we denote by $S^{p,q}=S_d^{p-1}\times S_d^{q-1}$.
With these coordinates we can write 
$$\cI_{h,I}=\int_0^{\infty}\int_0^{\infty}\int_{S^{p,q}}h(\frac{r_1\omega_1+r_2\omega_2}{T})\chi_I(r_1^d-r_2^d)r_1^{p-1}r_2^{q-1}d\omega_1d\omega_2dr_1dr_2.$$
Make a change of variable $s=r_1^d-r_2^d$  so that $r_1=(r_2^d+s)^{1/d}$ and $dr_1=\frac{ds}{d (r_2^d+s)^{\frac{d-1}{d}}}$.
With this change of variable, writing $r=\frac{r_2}{T}$ we get 
%$$\cI_{h,I}=\int_0^{1}\chi_I(u)\int_0^{\infty}\int_{S^{p,q}}h(\frac{\sqrt{r^2+u}\omega_1+r\omega_2}{T})(r^2+u)^{(p-2)/2}r^{q-1}d\omega_1d\omega_2drdu.$$
%Now make another change of variable $r\mapsto Tr$ to get 
$$ \cI_{h,I}=\frac{T^{n-d}}{d}\int_{0}^{1}\chi_I(s)\int_{0}^{\infty}\int_{S^{p,q}}h(\omega_1(r^d+\tfrac{s}{T^d})^{1/d}+\omega_2 r)(r^d+\tfrac{s}{T^d})^{\frac{p-d}{d}} r^{q-1}d\omega_1d\omega_2 drds.$$
For $0<r\leq1/T$ and $0\leq s\leq 1$, using the estimates $(r^d+\tfrac{s}{T^d})^{\frac{p-d}{d}}\ll \frac{1}{T^{p-d}}$ if $p\geq d$ and $(r^d+\frac{s}{T^d})^{\frac{p-d}{d}}\leq r^{p-d}$ if $p<d$ we can bound the contribution to this integral of the range $0<r\leq 1/T$ by
$O(|I|\|h\|_\infty)$ to get that 
\begin{align*}
 \cI_{h,I}&=\frac{T^{n-d}}{d}\int_{0}^{1}\chi_I(s)\int_{1/T}^{\infty}\int_{S^{p,q}}h(\omega_1(r^d+\tfrac{s}{T^d})^{1/d}+\omega_2 r)(1+\tfrac{s}{(rT)^d})^{\frac{p-d}{d}} r^{n-d-1}d\omega_1d\omega_2 drds\\
 &+O(|I|\|h\|_\infty).
 \end{align*}
Now, noting that for $r>1/T$ and $0\leq s\leq 1$ we have  $(r^d+\tfrac{s}{T^d})^{1/d}=r(1+\frac{s}{(rT)^d})^{1/d}=r+O(\frac{1}{r^{d-1}T^d})$, and
we can estimate 
$$h(\omega_1(r^d+\tfrac{s}{T^d})^{1/d}+\omega_2 r)=h((\omega_1+\omega_2)r)+
O\left(\frac{\cS_1(h)}{r^{d-1}T^{d}}\right) $$
to get that 
 \begin{align*}
 \cI_{h,I}&=\frac{T^{n-d}}{d}\int_{0}^{1}\chi_I(s)\int_{1/T}^{\infty}\int_{S^{p,q}}(h((\omega_1+\omega_2)r)(1+\tfrac{s}{(rT)^d})^{\frac{p-d}{d}} r^{n-d-1}d\omega_1d\omega_2 drds\\
 &+O\left(\cS_1(h) |I|(T^{n-2d}\int_{1/T}^{a} r^{n-2d}dr)\right) +O(|I|\|h\|_\infty),
 \end{align*}
 where for the first part in the error term we used the assumption that $h$ is supported on $B_a$ noting that $\|\omega_1(r^d+\tfrac{s}{T^d})^{1/d}+\omega_2 r\|_2\geq r$ for all $(\omega_1,\omega_2)\in S^{p,q}$.
%Now, for $p=d$ we have that
%$$\cI_{h,I}= \frac{|I|T^{n-d}}{d}\int_{1/T}^{\infty}\int_{S^{p,q}}h((\omega_1+\omega_2)r)d\omega_1 d\omega_2dr+O_{p,q}\bigg(\cS_1(h)|I| \frac{\log(aT)}{T}\bigg)+O_{p,q}(|I|\|h\|_\infty).$$
%\begin{align*}
%\cI_{h,I}&= |I|T\int_{1/T}^{\infty}\int_{S^{p,q}}h((\omega_1+\omega_2)r)d\omega_1 d\omega_2drdu\\
%&+O_{p,q}\bigg(\cS_1(h)|I| \frac{\log(aT)}{T}\bigg)+O_{p,q}(|I|\|h\|_\infty),
%\end{align*}
 Next, we can estimate $(1+\tfrac{s}{(rT)^d})^{\frac{p-d}{d}}=1+O(\frac{1}{(rT)^d})$ uniformly for $r> 1/T$ and $0\leq s\leq 1$ to get that 
 \begin{align*}
 \cI_{h,I}&=\frac{T^{n-d}}{d}\int_{0}^{1}\chi_I(s)\int_{0}^{\infty}\int_{S^{p,q}}(h((\omega_1+\omega_2)r)r^{n-d-1}d\omega_1d\omega_2 drds+O\left(|I|\|h\|_\infty\right)\\
 &+O\left(\|h\|_\infty  |I|T^{n-2d}\int_{1/T}^{a} r^{n-2d-1} dr\right)+O\left(\cS_1(h) |I|T^{n-2d}\int_{1/T}^{a} r^{n-2d}dr)\right) ,
 \end{align*}
 where we used that 
 $$\int_{0}^{1}\chi_I(s)\int_{0}^{1/T}\int_{S^{p,q}}h((\omega_1+\omega_2)r)r^{n-d-1}d\omega_1d\omega_2 drds\ll \|h\|_\infty |I| T^{d-n}.$$
Now estimate  
$$\int_{1/T}^{a} r^{n-2d-1} dr\ll \left\lbrace\begin{array}{ll} 
a^{n-2d} & n\geq 2d+1\\
\log(aT) & n=2d\\
T^{2d-n}  & n<2d,
\end{array}\right. $$
and similarly 
$$\int_{1/T}^{a} r^{n-2d} dr\ll \left\lbrace\begin{array}{ll} 
a^{n-2d+1} & n\geq 2d\\
\log(aT) & n=2d-1\\
T^{2d-n-1}  & n<2d-1
\end{array}\right. $$  
to get that 
\begin{align*}
 \cI_{h,I}&= J(h) |I|T^{n-d}+\left\{\begin{array}{ll}
O\left(\|h\|_\infty  |I| [aT]^{n-2d}\right)+O\left(\cS_1(h)|I|a [aT]^{n-2d}\right)  & n\geq 2d+1\\
O\left(\|h\|_\infty  |I| \log(aT)\right)+O\left(\cS_1(h)|I|a\right) & n=2d\\
O\left(\|h\|_\infty  |I| \right)+O\left(\cS_1(h) |I| \frac{\log(aT)}{T}\right)   & n=2d-1\\
O\left(\|h\|_\infty  |I| \right)+ O\left(\cS_1(h)|I| \frac{1}{T}\right)& n<2d-1.
 \end{array}\right.
\end{align*}
Noting that for $n\geq 2d+1$ the first term is bounded by the second term gives our result.
\end{proof}

We can now unsmooth to obtain the following.
\begin{proof}[Proof of Theorem \ref{t:vol}]
We use the notation $S^{p,q}=S_d^{p-1}\times S_d^{q-1}$ as before. First assume that $N=1$  so that $I\subseteq [-1,1]$. Let $h_0=\chi_{B_1}$ denote the indicator function of the unit ball and for small $\delta>0$
let $h_\delta^\pm$ be smooth functions taking values in $[0,1]$ approximating $h_0$ in the sense that $h_\delta^-(v)=\left\{\begin{array}{ll} 1 & \|v\|\leq1-\delta\\ 0 & \|v\|\geq 1\end{array}\right.$ and similarly $h_\delta^+(v)=\left\{\begin{array}{ll} 1 & \|v\|\leq1\\ 0 & \|v\|\geq 1+\delta \end{array}\right.$, and we can choose them so that $\cS_1(h_\delta^\pm)\ll \delta^{-1}$. 

Now recall that any $F\in Y_{p,q}^{(d)}$ satisfies that $F(v)=F_0(vg)$ for some $g\in G$ to get that 
$$\vol(F^{-1}(I)\cap B_T)=\int_{\R^n} h_0(\frac{v}{T})\chi_I(F_0(vg))dv=\int_{\R^n} h_0(\frac{v g^{-1}}{T})\chi_I(F_0(v))dv,$$
and we can approximate it from above and below by
$\int_{\R^n} h_\delta^\pm (\frac{v g^{-1}}{T})\chi_I(F_0(v))dv$. 

Let $h_{\delta,g}^\pm(v)=h_{\delta}^\pm(v g^{-1})$, so that $h_{\delta,g}$ is supported on $B_{2\|g\|}$ and satisfy
$\cS_1( h_{\delta,g}^\pm)=O(\|g\|\delta^{-1})$. Now applying Lemma \ref{l:smoothvol} to these functions gives us that for any $T>1$
\begin{equation}\label{e:smoothed}
\int_{\R^n} h_{\delta,g}^\pm (\frac{v}{T})\chi_I(F_0(v))dv=J(h_{\delta,g}^\pm)|I| T^{n-d}+
\left\{\begin{array}{ll}
O\left(\|g\|^{n-2d+2} \delta^{-1}|I|  T^{n-2d}\right)  & n\geq 2d+1\\
O\left(  |I| \log(2\|g\|T)\right)+O\left(\|g\|^2\delta^{-1}|I|| \right) & n=2d\\
O\left(  |I| \right)+O\left( \|g\| \delta^{-1} |I| \frac{\log(2\|g\|T)}{T}\right)   & n=2d-1\\
O\left(  |I| \right)+ O\left(\frac{\|g\| |I| }{\delta T}\right)& n<2d-1.
 \end{array}\right.
\end{equation}
%where the term $\log(\|g\|T)$ is only needed when $n=3$.

Next, let $h_{0,g}(v)=h_0(vg^{-1})$ and note that $J(h_{\delta,g}^-)\leq J(h_{0,g})\leq J(h_{\delta,g}^+)$ and that 
$$J(h_{\delta,g}^+)-J(h_{\delta,g}^-)=\int_{S^{p,q}} \int_0^\infty \left(h_\delta^+(r(\omega_1+\omega_2)g^{-1})-h_\delta^-(r(\omega_1+\omega_2)g^{-1})\right)r^{n-d-1}drd\omega_1d\omega_2.$$
Since $h_\delta^+(r(\omega_1+\omega_2)g^{-1})-h_\delta^-(r(\omega_1+\omega_2)g^{-1})=0$ unless $\|r(\omega_1+\omega_2)g^{-1}\|\in (1-\delta,1+\delta)$, we can bound 
\begin{align*} 
|J(h_{\delta,g}^\pm)-J(h_{0,g})|&\leq \int_{S^{p,q}}\int_{\frac{1-\delta}{\|(\omega_1+\omega_2)g^{-1}\|}}^{\frac{1+\delta}{\|(\omega_1+\omega_2)g^{-1}\|}}r^{n-d-1}dr d\omega_1d\omega_2\\
&\ll \delta  \int_{S^{p,q}}\|(\omega_1+\omega_2)g^{-1}\|^{d-n}d\omega_1d\omega_2\ll \|g\|^{n-d}\delta,
\end{align*}
where in the last step we used that $\|(\omega_1+\omega_2)g^{-1}\|\gg \|g\|^{-1}$ for any $(\omega_1,\omega_2)\in S^{p,q}$.

Plugging this estimate back in \eqref{e:smoothed} and taking 
$$\delta=\left\{\begin{array}{ll} T^{-d/2} & n\geq 2d\\
T^{\frac{d-n-1}{2}} & d<n\leq 2d-1\end{array}\right. $$ 
we get that for any $T>2$, $\vol(F^{-1}(I)\cap B_T)$ is bounded both from above and from below by
$$J(h_{0,g})|I|T^{n-d}+\left\lbrace \begin{array}{ll} 
O_g(|I| |T^{n-3d/2})& n\geq 2d\\ 
O_g(|I| T^{\frac{n-d-1}{2}}\log(T) )& n=2d-1\\
O_g(|I| T^{\frac{n-d-1}{2}} )& n<2d-1,\\
\end{array}\right.
$$
where for the $n=2d-1$ case we used the estimate $\log(2\|g\|T)=O_g(\log(T))$ for all $T>2$. Setting $c_F=J(h_{0,g})$ concludes the proof for the case of $I\subseteq [-1,1]$, where we note that the implied constant is bounded by some power of $\|g\|$ and is hence uniform for $g$ in compact sets.
%$$J(h_{0,g})|I|T^{n-d}+O(\|g\|^{n-d} \|I| T^{n-d-1})+\left\{\begin{array}{ll}
%O\bigg(\|g\|^{n-2d+2} |I|  T^{n+1-2d}\bigg)  & n\geq 2d+1\\
%O\bigg(  |I| \log(\|g\|T)\bigg)+O\bigg(\|g\|^2 |I| T \bigg) & n=2d\\
%O\bigg( \|g\|  |I| \log(\|g\|T)\bigg)   & n=2d-1\\
% O\bigg(\|g\| |I| \bigg)& n<2d-1
% \end{array}\right\}
%$$
%$$J(h_{0,g})|I|T^{n-2}+O_n(\delta \|g\|^{n-2} \|I| T^{n-2})+O_{p,q}(|I| \delta^{-1}\|g\|^{n-2} T^{n-4}\log(\|g\|T)),$$
%We then conclude the proof in this case by taking $\delta=1/T$ and taking the constant $c_Q=J(h_{0,g})$, and noting that $\log(\|g\|T)\leq \|g\|^2T$ and $1\leq \log(\|g\|T)$ for $T> 3$.

Finally, we consider the general case of an interval $I\subseteq [-N,N]$. Denote by $I'=\frac{1}{N}I\subseteq [-1,1]$ and note that for any $v\in \R^n$ we have that $F(v)\in I$ if and only if $F( N^{-1/d} v)\in I'$, so that writing $v=N^{1/d}u$ we have that 
$$F^{-1}(I)\cap B_T=\{v\in \R^n: F(v)\in I,\; \|v\|\leq T\}=\{N^{1/d} u: F(u)\in I',\; \|u\|\leq TN^{-1/d}\}.$$
Since we assume $T> 2N^{1/d}$ we can apply the previous result to get that indeed
\begin{align*}
\vol(F^{-1}(I)\cap B_T)&=N^{n/d}\vol(\{u\in \R^n: F(u)\in I',\; \|u\|\leq TN^{-1/d}\})\\
&=c_F |I|T^{n-d}+\left\lbrace \begin{array}{ll} 
O_g(|I| |T^{n-3d/2}  N^{1/2})& n\geq 2d\\ 
O_g(|I|  T^{\frac{n-d-1}{2}}N^{\frac{n+1-d}{2d}}\log(T) )& n=2d-1\\
O_g(|I| T^{\frac{n-d-1}{2}} N^{\frac{n+1-d}{2d}})& n<2d-1,\\
\end{array}\right.
\end{align*}
where for the $n=2d-1$ case we used that $\log(TN^{-1/d})\leq \log(T)$ for $N\geq 1$. Finally, we conclude the proof by noting that for $T>2N^{1/d}$, $T^{n-3d/2}N^{1/2}\ll T^{n-d-1}N^{1/d}$ and $T^{\frac{n-d-1}{2}}N^{\frac{n+1-d}{2d}}\ll T^{n-d-1}N^{1/d}$.
\end{proof}

\section{Approximating a single point}\label{sec:single}
In this section, we prove Theorem \ref{thm:singleapprox} by reducing it into a lattice point counting problem, and more precisely to a discrepancy estimate.

As mentioned in the introduction, for $F=g\cdot F_0$ with $g\in G$ and for any measurable subset $I\subset \R$ we have that the counting function
\begin{equation}\label{equ:transtolattice}
\cN_F(I,t):=\#\left(\Z^n\cap F^{-1}(I)\cap B_t\right)=\#\left(\Z^ng\cap F_0^{-1}(I)\cap B_tg\right)
\end{equation}
counts the number of lattice points of $\Lambda=\Z^ng$ that lie inside the set $F_0^{-1}(I)\cap B_tg$. To simplify notation, for any $g\in G$, and subset $I\subset \R$ and any $t>0$ we denote by \begin{equation}\label{e:Ag}
A_{g,I,t}:=F_0^{-1}(I)\cap B_tg.
\end{equation}
In view of this relation, to get a power saving asymptotic formula for the counting function $\cN_F(I,t)$, we first prove a power saving asymptotic bound for the discrepancy $D(\Z^ng, A_{g,I,t})$.
\begin{thm}\label{thm:latticecounting}
%Let  $n=p+q>d$ and $0<\kappa<n-d$ be as in Theorem \ref{thm:singleapprox}. Let $\{I_t\}_{t\geq 1}$ be a decreasing family of bounded measurable subsets of $\R$ with measures $|I_t|=ct^{-\kappa}$ with some $c>0$. 
Keep the assumptions as in Theorem \ref{thm:singleapprox}. Then there exists some $\delta\in (0,1)$ such that for $\mu$-a.e. $g\in G$ there exists $t_g>0$ such that for all $t\geq t_g$
$$D(\Z^ng,A_{g,I_t,t})<\vol(A_{g,I_t,t})^{\delta}.$$
\end{thm}
\begin{proof}
%For the readers convenience we recall here that $\delta_0=1-\frac{2}{n^2+n+4}=\frac{d_n+2}{d_n+3}$ with $d_n=\frac{(n+2)(n-1)}{2}$ as in Lemma \ref{l:cover}.
Fix a compact set $\cK\subset G$ and a sequence $\{t_k=k^{\alpha}\}_{k\in\N}$ with the exponent $\alpha>\max\{1,\frac{1}{n-d-\kappa}\}$ depending on $\kappa$ to be determined. Let $\delta_0=1-\frac{1}{\alpha(n-d-\kappa)}$. For any $t>0$ and $\delta\in (\delta_0,1)$ consider the set $\cB_t\subset \cK$ defined by
$$\cB_t=\{g\in \cK\ |\ D(\Z^ng, A_{g,I_t,t}) \geq \vol( A_{g,I_t,t})^{\delta}\}.$$
We will show that $\limsup_{t\to\infty}\cB_t$ is a null set (this will imply that for $\mu$-a.e. $g\in \cK$ we have that $D(\Z^ng, A_{g,I_t,t}) < \vol( A_{g,I_t,t})^{\delta}$ for all sufficiently large $t$ and since this holds for any compact set $\cK$ this will conclude the proof).
 
Now, since the sequence $\{t_k\}_{k\in \N}$ is unbounded,
$$\limsup_{t\to\infty}\cB_t=\bigcap_{T>0}\bigcup_{t\geq T}\cB_t=\bigcap_{m\in \N}\bigcup_{k\geq m}\bigcup_{t_k\leq t< t_{k+1}}\cB_t.$$
and hence
$$\mu\left(\limsup_{t\to\infty}\cB_t\right)=\lim_{m\to\infty}\mu\left(\bigcup_{k\geq m}\bigcup_{t_k\leq t< t_{k+1}}\cB_t\right)\leq \lim_{m\to\infty}\sum_{k=m}^{\infty}\mu\left(\bigcup_{t_k\leq t< t_{k+1}}\cB_t\right).$$
We thus need to estimate $a_k=\mu\left(\bigcup_{t_k\leq t< t_{k+1}}\cB_t\right)$ and show that the series $\sum_k a_k$ is summable. 

Now for $\epsilon_k=\frac{1}{k}$, by Lemma \ref{prop:covering}, for any $k\geq 1$ there exists a finite subset 
$\cI_k:=\cI_{\epsilon_k}\subset \cK$ with $\# \cI_k=O_{\cK}(k^{d_n})$ such that $\cK\subset \bigcup_{h\in \cI_k}\cO_{\epsilon_k}h$ where $d_n=\frac{(n+2)(n-1)}{2}$. 
Thus for $g\in \cK$, there exists some $h\in \cI_k$ and $g' \in \cO_{\epsilon_k}$ such that $g=g' h$. Then by \eqref{equ:covering} for any $t_k\leq t< t_{k+1}$ 
$$B_{(1-\epsilon_k)t_k}h\subset B_{t_k}g' h\subset B_{t}g\subset B_{t_{k+1}}g' h \subset B_{(1+\epsilon_k)t_{k+1}}h,$$
and since $I_{t_{k+1}}\subseteq I_t\subseteq I_{t_k}$ we get that 
$$\underline{A}_{k,h}\subset A_{g,I_t,t}\subset \overline{A}_{k,h},$$ 
where
$$\underline{A}_{k,h}=F_0^{-1}(I_{t_{k+1}})\cap B_{(1-\epsilon_k)t_k}h\quad \textrm{and}\quad \overline{A}_{k,h}=F_0^{-1}(I_{t_k})\cap B_{(1+\epsilon_k)t_{k+1}}h.$$
Hence using the interpolation Lemma \ref{l:inter}, we get that for any  $g\in \bigcup_{t_k\leq t< t_{k+1}}\cB_t$ there is $h\in \cI_k$ such that 
\begin{align*}
\max\left\{D(\Z^ng,\underline{A}_{k,h}),D(\Z^ng,\overline{A}_{k,h})\right\}&\geq D(\Z^ng, A_{g,I_t,t})-\vol((\overline{A}_{k,h}\setminus \underline{A}_{k,h}))\\
&\geq \vol(A_{g,I_t,t})^{\delta}-\vol((\overline{A}_{k,h}\setminus \underline{A}_{k,h}))\\
&\geq \vol(\underline{A}_{k,h})^{\delta}-\vol((\overline{A}_{k,h}\setminus \underline{A}_{k,h}))
\end{align*}
implying that 
\begin{equation}\label{equ:relation}
\bigcup_{t_k\leq t< t_{k+1}}\cB_t\subset \bigcup_{h\in \cI_k}\cC_{k,h},
\end{equation}
where 
$\cC_{k,h}=\cM_{\underline{A}_{k,h},T_{k,h}}\cup \cM_{\overline{A}_{k,h},T_{k,h}}$ with
$T_{k,h}=\vol(\underline{A}_{k,h})^{\delta}-\vol(\overline{A}_{k,h}\setminus \underline{A}_{k,h})$, and $\cM_{A,T}=\cM^{(\cK)}_{A,T}$ defined in \eqref{e:MAT}.
Now applying Lemma  \ref{l:MAT} we can bound 
\begin{equation}\label{equ:estimate1}
\mu(\cC_{k,h})\leq \mu(\cM_{\underline{A}_{k,h},T_{k,h}})+\mu(\cM_{\overline{A}_{k,h},T_{k,h}})\ll_{\cK}\frac{\vol(\overline{A}_{k,h})}{T_{k,h}^2}.
\end{equation}
Since $\{I_t\}_{t>0}$ is bounded, there exists some $N_0>0$ such that $I_t\subset [-N_0,N_0]$ for all $t>0$, and hence for all $k\geq k_0$ sufficiently large
%$$(1-\epsilon_k)t_k>\sqrt{N}\quad \textrm{and}\quad (1+\epsilon_k)t_{k+1}>\sqrt{N}.$$
we can apply Theorem \ref{t:vol} to $\overline{A}_{k,h}$ (or more precisely to $\overline{A}_{k,h}h^{-1}$ having the same volume)  which, recalling that $|I_t|=ct^{-\kappa}$, gives
$$\vol(\overline{A}_{k,h})=c_h ct_k^{-\kappa}((1+\epsilon_k)t_{k+1})^{n-d}+O_{\cK,c}\left(t_k^{-\kappa}t_{k+1}^{n-d-1}\log(t_{k+1})\right),$$
with $c_h=c_{h\cdot F_0}$ from Theorem \ref{t:vol}.  Next, plug in $t_k=k^{\alpha}$ and use the estimates $(1+\epsilon_k)^{n-d}=1+O(\frac{1}{k})$ and $t_{k+1}^{n-d}=(k+1)^{\alpha(n-d)}=k^{\alpha(n-d)}(1+O_{\kappa}(\frac{1}{k}))$ to get
$$\vol(\overline{A}_{k,h})=c_hc k^{\alpha(n-d-\kappa)}+O_{\cK,c,\kappa}\left(k^{\alpha(n-d-\kappa)-1}+k^{\alpha(n-d-1-\kappa)}\log(k+1)^{\alpha}\right).$$
Since by assumption $\alpha>1$ we have $\alpha(n-d-1-\kappa)<\alpha(n-d-\kappa)-1$ so that
\begin{equation}\label{equ:upperes}
\vol(\overline{A}_{k,h})=c_hc k^{\alpha(n-d-\kappa)}+O_{\cK,c,\kappa}\left(k^{\alpha(n-d-\kappa)-1}\right).
\end{equation}
Similarly, applying Theorem \ref{t:vol} to $\underline{A}_{k,h}$ for $k>k_0$ sufficiently large we get
\begin{equation}\label{equ:loweres}
\vol(\underline{A}_{k,h})=c_hc k^{\alpha(n-d-\kappa)}+O_{\cK,c,\kappa}\left(k^{\alpha(n-d-\kappa)-1}\right),%\asymp_{\cK,p,q,c}k^{\alpha(n-2-\kappa)}.
\end{equation}
and hence, 
$$\vol(\overline{A}_{k,h}
\setminus \underline{A}_{k,h})\ll_{\cK,c,\kappa} k^{\alpha(n-d-\kappa)-1}.$$
We also have for $k>k_0$ sufficiently large  
$$\vol(\underline{A}_{k,h})^{\delta}\asymp_{\cK,c} k^{\delta \alpha(n-d-\kappa)}.$$
Moreover, since we assume $\delta> 1-\frac{1}{\alpha(n-d-\kappa)}$ then $\delta\alpha(n-d-\kappa)> \alpha(n-d-\kappa)-1$ and hence
$$T_{k,h}=\vol(\underline{A}_{k,h})^{\delta}-\vol(\overline{A}_{k,h}\setminus \underline{A}_{k,h})\asymp_{\cK,c,\kappa} k^{\delta \alpha(n-d-\kappa)}.$$
Combining the above estimates with \eqref{equ:estimate1} we get that 
$$\mu(\cC_{k,h})\ll_{\cK,c,\kappa} \frac{\vol(\overline{A}_{k,h})}{T_{k,h}^2}\asymp_{\cK,c,\kappa} k^{\alpha(n-d-\kappa)(1-2\delta)}.$$
We can then use \eqref{equ:relation} and the fact $\#\cI_k=O_{\cK}(k^{d_n})$ to get for $k>k_0$ sufficiently large
$$a_k\leq \mu(\bigcup_{h\in \cI_k}\cC_{k,h})\leq \sum_{h\in \cI_k}\mu(\cC_{k,h})\ll_{\cK,c,\kappa}  \frac{1}{k^{\alpha(n-d-\kappa)(2\delta-1)-d_n}}.$$
Now, the assumption $\delta> 1-\frac{1}{\alpha(n-d-\kappa)}$ also implies that $\alpha(n-d-\kappa)(2\delta-1)-d_n>\alpha(n-d-\kappa)-2-d_n$, so taking $\alpha=\frac{d_n+3}{n-d-\kappa}>\max\{1,\frac{1}{n-d-\kappa}\}$ we get that $\alpha(n-d-\kappa)(2\delta-1)-d_n>1$ showing that the series $\sum_ka_k$ is summable as needed.
\end{proof}

%%\section{Approximating a single point}
%%In this section, prove to Theorem \ref{thm:singleapprox} by reducing it into a lattice point counting problem, and more precisely to a discrepancy estimate.
%%
%%Note that for $Q=g\cdot Q_0$ with $g\in G$ and for any measurable subset $I\subset \R$ we have that $Q^{-1}(I)=Q_0^{-1}(I)g^{-1}$ and hence the counting function
%%\begin{equation}\label{equ:transtolattice}
%%\cN_Q(I,t):=\#\left(\Z^n\cap Q^{-1}(I)\cap B_t\right)=\#\left(\Z^ng\cap Q_0^{-1}(I)\cap B_tg\right)
%%\end{equation}
%%is counting the number of lattice points of $\Lambda=\Z^ng$ that lie inside the set $Q_0^{-1}(I)\cap B_tg$. To simplify notation, for any $g\in G$, and subset $I\subset \R$ and any $t>0$ we denote by \begin{equation}\label{e:Ag}
%%A_{g,I,t}:=Q_0^{-1}(I)\cap B_tg.
%%\end{equation}
%%In view of this relation, to get a power saving asymptotic formula for the counting function $\cN_Q(I,t)$, we first prove a power saving asymptotic bound for the discrepancy $D(\Z^ng, A_{g,I,t})$.
%%\begin{thm}\label{thm:latticecounting}
%eorem \ref{t:vol} to $\overline{A}_{k,h}$ which, recalling that $|I_t|=ct^{-\kappa}$, gives

We can now use this discrepancy bound to conclude.
\begin{proof}[Proof of Theorem \ref{thm:singleapprox}]
Fix $\delta\in (\delta_0,1)$ be as above with $\delta_0=1-\frac{1}{\alpha(n-d-\kappa)}=\frac{d_n+2}{d_n+3}$, and let $0<\nu<(1-\delta)(n-d-\kappa)$. 
Then by Theorem \ref{thm:latticecounting} for $\mu$-a.e. $g\in G$ we have that $D(\Z^n g, A_{g,I_t,t})<\vol(A_{g,I_t,t})^{\delta}$ for all sufficiently large $t$, and hence for $F=g\cdot F_0$ with $g$ as above and all sufficiently large $t$,
\begin{align*}
\left|\cN_F(I_t,t)-c_F|I_t|t^{n-d}\right|&\leq D(\Z^ng,A_{g,I_t,t})+ |\vol(A_{g,I_t,t})-c_F|I_t|t^{n-d}|\\
&< \vol(A_{g,I_t,t})^{\delta}+O_{g,c}(t^{n-d-1-\kappa}\log(t))\\
&<(2c_Ft^{n-d-\kappa})^{\delta}+O_{g,c}(t^{n-d-1-\kappa}\log(t))< t^{n-d-\kappa-\nu}.\qedhere
\end{align*} 
\end{proof}

\section{Uniform approximation}\label{sec:uniform}
We now use similar ideas to give a uniform bound for the discrepancy for all intervals at once. We first prove a preliminary uniform bound for the discrepancy for all intervals of a fixed length.
\begin{thm}\label{thm:simuapprox}
Keep the assumptions as in Theorem \ref{thm:sim2}. Then there exists some $\delta\in (0,1)$ such that for almost all $F\in Y^{(d)}_{p,q}$ there is $t_F>0$ such that for all $t>t_F$ and for all intervals $I\subset [-N(t),N(t)]$ with $|I|=t^{-\kappa}$ we have
\begin{equation}\label{equ:counting2}
|\cN_F(I,t)-c_Ft^{n-d-\kappa}|< t^{\delta(n-d-\kappa)},
\end{equation}
where $c_F$ is as in Theorem \ref{t:vol}.
\end{thm}
\begin{proof}
Fix a compact subset $\cK\subset G$, and a sequence $\{t_k=k^{\alpha}\}_{k\in \N}$ with $\alpha>\frac{1}{n-d-\kappa}$ some large number depending on $\kappa$ and $\eta$ to be determined. Take $\delta$ such that $\delta\in(1-\frac{1}{\alpha(n-d-\kappa)},1)$. For any $t>0$ we define $\cB_t\subset \cK$ 
$$\cB_t=\{g\in \cK: \;\exists I\subseteq [-N(t),N(t)],\; |I|=t^{-\kappa},\;  \ D(\Z^ng, A_{g,I,t})\geq \vol(A_{g,I,t})^{\delta}\},$$
where $A_{g,I,t}=F_0^{-1}(I)\cap B_tg$ as before. As in the proof of Theorem \ref{thm:latticecounting} %and Theorem \ref{thm:singleapprox} 
it suffices to show that the series $\sum_k a_k$ with $a_k=\mu\left(\bigcup_{t_k\leq t<t_{k+1}}\cB_t\right)$ is summable.
To do that we will bound the set $\bigcup_{t_k\leq t<t_{k+1}}\cB_t$ by a nicer set for which we have good control on the measure.
 
%Suppose $g\in \bigcup_{t_k\leq t<t_{k+1}}\cB_t$, then by definition there exists some $t_k\leq t<t_{k+1}$ and some interval $I\subset (-N(t),N(t))$ with $|I|=t^{-\kappa}$ such that
%\begin{equation}\label{equ:discrepancy}
%D(\Z^ng, A_{g,I,t})\geq |A_{g,I,t}|^{\delta}.
%\end{equation}
First, for any $t_k\leq t<t_{k+1}$ we reduce the collection of all intervals in  $[-N(t),N(t)]$  into a finite discrete collection of intervals.
Let $\beta=\kappa+\frac{1}{\alpha}$ and let $-N\left(t_{k+1}\right)=\xi_{k,0}<\xi_{k,1}<\cdots<\xi_{k,M(k)}=N\left(t_{k+1}\right)$ be a $t_{k+1}^{-\beta}$-dense partition of the interval $\left[-N(t_{k+1}),N(t_{k+1})\right]$ %That is, the first $M(k)-1$ subintervals are of length $t_{k+1}^{-\beta}$ and the last interval is of length smaller or equal to $t_{k+1}^{-\beta}$. 
so that
 $M(k)%=\left \lceil{\frac{2N(t_{k+1})}{t_{k+1}^{-\beta}/2}}\right \rceil
 \asymp N(t_{k+1})t_{k+1}^{\beta}.$
% For later use we also note that $\beta$ satisfies the relation $\alpha\beta=\alpha \kappa+1$. 

Now for any interval $I\subset [-N(t),N(t)]$ with $t_k\leq t<t_{k+1}$, since its center point $\xi$ satisfies $\xi\in I\subset [-N(t),N(t)]\subset [-N(t_{k+1}),N(t_{k+1})]$, there exists some $0\leq i<M(k)$ such that $\xi_{k,i}\leq \xi<\xi_{k,i+1}$. Note that for $I=(\xi-\frac{t^{-\kappa}}{2},\xi+\frac{t^{-\kappa}}{2})$, since $t_k\leq t<t_{k+1}$ and $\xi_{k,i}\leq \xi<\xi_{k,i+1}$ we have 
\begin{equation}\label{equ:containment2}
(\xi_{k,i+1}-\frac{t_{k+1}^{-\kappa}}{2},\xi_{k,i}+\frac{t_{k+1}^{-\kappa}}{2})\subset I\subset (\xi_{k,i}-\frac{t_k^{-\kappa}}{2},\xi_{k,i+1}+\frac{t_k^{-\kappa}}{2}).
\end{equation}
%We note that the above first interval is nonempty since $\xi_{k,i+1}-\xi_{k,i}\leq t_{k+1}^{-\beta}<t_{k+1}^{-\kappa}$, where the second inequality comes from the fact that $\beta=\kappa+\frac{1}{\alpha}>\kappa$. 

Next, let $\epsilon_k=\frac{1}{k}$ and again by Lemma \ref{prop:covering} we know that for any $k\geq 1$, there exists a finite subset $\cI_k:=\cI_{\epsilon_k}\subset \cK$ with $\#\cI_k=O_{\cK}(k^{d_n})$ such that $\cK\subset \bigcup_{h\in \cI_k}\cO_{\epsilon_k}h$. Thus for any $g\in \bigcup_{t_k\leq t<t_{k+1}}\cB_t$, there exists some $g'\in \cO_{\epsilon_k}$ and $h\in \cI_k$ such that $g=g' h$. Then by \eqref{equ:covering} for $t_k\leq t<t_{k+1}$ we have $B_{(1-\epsilon_k)t_k}h\subset B_tg\subset B_{(1+\epsilon_k)t_{k+1}}h$. Combining this with \eqref{equ:containment2} we get $\underline{A}_{k,i,h}\subset A_{g,I,t}\subset \overline{A}_{k,i,h}$ where
$$\underline{A}_{k,i,h}=F_0^{-1}(\xi_{k,i+1}-\frac{t_{k+1}^{-\kappa}}{2},\xi_{k,i}+\frac{t_{k+1}^{-\kappa}}{2})\cap B_{(1-\epsilon_k)t_k}h$$
and 
$$\overline{A}_{k,i,h}=F_0^{-1}(\xi_{k,i}-\frac{t_k^{-\kappa}}{2},\xi_{k,i+1}+\frac{t_k^{-\kappa}}{2})\cap B_{(1+\epsilon_k)t_{k+1}}h.$$
Then by Lemma \ref{l:inter} for any $g\in \bigcup_{t_k\leq t<t_{k+1}}\cB_t$  there exist some $0\leq i< M(k)$ and $h\in \cI_k$ such that
\begin{equation}\label{equ:interpo2}
\max\left\{D(\Z^ng,\underline{A}_{k,i,h}),D(\Z^ng,\overline{A}_{k,i,h})\right\}\geq \vol(\underline{A}_{k,i,h})^{\delta}-\vol(\overline{A}_{k,i,h}\setminus \underline{A}_{k,i,h}).
\end{equation}
Now for any $0\leq i<M(k)$ and $h\in \cI_k$ let 
$\cC_{k,i,h}=\cM^{(\cK)}_{\underline{A}_{k,i,h},T_{k,i,h}}\cup \cM^{(\cK)}_{\overline{A}_{k,i,h},T_{k,i,h}}$ with
$$T_{k,i,h}=\vol(\underline{A}_{k,i,h})^{\delta}-\vol(\overline{A}_{k,i,h}\setminus \underline{A}_{k,i,h}),$$
so that 
\begin{equation}\label{equ:relations}
\bigcup_{t_k\leq t<t_{k+1}}\cB_t\subset \bigcup_{h\in \cI_k}\bigcup_{0\leq i<M(k)}\cC_{k,i,h}.
\end{equation}
By Lemma  \ref{l:MAT} we can bound
\begin{equation}\label{equ:estimate2}
\mu(\cC_{k,i,h})\ll_{\cK}\frac{\vol\left(\overline{A}_{k,i,h}\right)}{T_{k,i,h}^2}.
\end{equation}
Now, we take
\begin{equation}\label{equ:volumern}
\alpha> \max\left\{\frac{d}{d-\eta}, \frac{d_n+4}{n-d-2\kappa-\eta} \right\}.
\end{equation}
We note that since $\kappa\in (0,\frac{n-d-\eta}{2})$, the second value in the above set is positive (and larger than $\frac{1}{n-d-\kappa}$). We can now use \eqref{equ:relations} and \eqref{equ:estimate2} to give estimates for $\mu\left(\bigcup_{t_k\leq t<t_{k+1}}\cB_t\right)$.

First, since $N(t)=O(t^{\eta})$ with $\eta\in [0,\min\{d,n-d\})\subset [0,d)$, for $t_k=k^{\alpha}$ and $\epsilon_k=\frac{1}{k}$, there exists $k_0>0$ such that for any $k>k_0$,
$$(1-\epsilon_k)t_k>2(N(t_{k+1})+1)^{1/d}\quad \textrm{and}\quad (1+\epsilon_k)t_{k+1}>2(N(t_{k+1})+1)^{1/d}.$$
Moreover, note that for each $0\leq i< M(k)$ the intervals $(\xi_{k,i+1}-\frac{t_{k+1}^{-\kappa}}{2},\xi_{k,i}+\frac{t_{k+1}^{-\kappa}}{2})$ and $(\xi_{k,i+1}-\frac{t_{k+1}^{-\kappa}}{2},\xi_{k,i}+\frac{t_{k+1}^{-\kappa}}{2})$ are all contained in $[-N(t_{k+1})-1,N(t_{k+1})+1]$. Hence for any $k>k_0$ we can apply Theorem \ref{t:vol} to $\overline{A}_{k,i,h}$ to get
$$\vol(\overline{A}_{k,i,h})=c_h\left(t_k^{-\kappa}+(\xi_{k,i+1}-\xi_{k,i})\right)(1+\epsilon_k)^{n-d}t_{k+1}^{n-d}+O_{\cK}\left(t_k^{-\kappa}t_{k+1}^{n-d-1}N(t_{k+1})^{1/d}\log(t_{k+1})\right).$$
%, and for the estimate in the error term we used that $0<\xi_{k,i+1}-\xi_{k,i}\leq t_{k+1}^{-\beta}=o(t_{k+1}^{-\kappa})$ (since $\beta=\kappa+\frac{1}{\alpha_n}>\kappa$) and $1+\epsilon_k\asymp 1$. 
Now,  using the assumption $N(t)=O(t^{\eta})$ and the estimates $\xi_{k,i+1}-\xi_{k,i}\leq t_{k+1}^{-\beta}$, $t_{k+1}\asymp k^{\alpha}$, $(1+\epsilon_k)^{n-d}=1+O(1/k)$ and $t_{k+1}^{n-d}=(k+1)^{\alpha(n-d)}= k^{\alpha(n-d)}(1+O_{\kappa,\eta}(1/k))$ and the relation $\alpha\beta=\alpha\kappa+1$ we can get
\begin{align*}
\vol(\overline{A}_{k,i,h})%&=c_h\left(t_k^{-\kappa}+O(t_{k+1}^{-\beta})\right)(1+O(1/k))t_{k+1}^{n-d}+O_{\cK}\left((t_k^{-\kappa}t_{k+1}^{n-d-1}N(t_{k+1})^{1/d}\log(t_{k})\right),\\
                              % &=c_ht_k^{-\kappa}t_{k+1}^{n-d}+O_{\cK}\left((k^{-1}t_k^{-\kappa}+t_{k+1}^{-\beta})t_{k+1}^{n-d}+k^{\alpha(n-d-1-\kappa+\frac{\eta}{d})}\log(k)\right)\\
                               %&=c_h k^{\alpha(n-d-\kappa)}(1+O_{\kappa,\eta}(1/k))+O_{\cK,\kappa,\eta}\left( k^{\alpha(n-d-\kappa)-1}+k^{\alpha(n-d-1-\kappa+\frac{\eta}{d})}\log(k)\right)\\
                               &= c_h k^{\alpha(n-d-\kappa)}+O_{\cK,\kappa,\eta}\left( k^{\alpha(n-d-\kappa)-1}+k^{\alpha(n-d-1-\kappa+\frac{\eta}{d})}\log(k)\right).
\end{align*}
Since $\alpha> \frac{d}{d-\eta}$ we have that $\alpha(n-d-1-\kappa+\frac{\eta}{d})<\alpha(n-d-\kappa)-1$,
%$$\max (\alpha(n-3-\kappa+\frac{n-2-\kappa}{n}-\frac{\epsilon}{2}), \alpha(n-2-\kappa-\frac{n\epsilon}{2}))\leq \alpha(n-2-\kappa)-1.$$ 
so that
\begin{equation}\label{equ:as2}
\vol(\overline{A}_{k,i,h})=c_h k^{\alpha(n-d-\kappa)}+ O_{\cK,\kappa,\eta}(k^{\alpha(n-d-\kappa)-1}).%\asymp_{\cK}  k^{\alpha(n-2-\kappa)}.
\end{equation}
%where we used the fact that 
%$$t_k^{-\kappa}t_{k+1}^{n-3+\frac{\eta}{2}}+t_{k+1}^{\frac{n\eta}{2}}\asymp k^{\alpha\alpha(n-3-\kappa+\frac{\eta}{2})}+k^{\frac{\alphan\eta}{2}}=o( k^{\alpha(n-2-\kappa)})$$
%which is implied by \eqref{equ:volumern}. 
Similarly, we can apply Theorem \ref{t:vol} to $\underline{A}_{k,i,h}$ to get for any $k>k_0$ sufficiently large
\begin{equation}\label{equ:as1}
\vol(\underline{A}_{k,i,h})=c_h k^{\alpha(n-d-\kappa)}+ O_{\cK,\kappa,\eta}(k^{\alpha(n-d-\kappa)-1}).%\asymp_{\cK}  k^{\alpha(n-2-\kappa)}.
\end{equation}

Using these estimates and that $\delta\alpha(n-d-\kappa)>\alpha(n-d-\kappa)-1$ we get that
%\begin{equation}%\label{e:alpha1}
%\alpha<\frac{1}{(1-\delta)(n-d-\kappa)}
%\end{equation} we get that 
$$T_{k,i,h}\asymp_{\cK,\kappa,\eta} \vol(\underline{A}_{k,i,h})^{\delta}\asymp_{\cK,\kappa,\eta} k^{\delta \alpha(n-d-\kappa)},$$
so that 
$$\mu(\cC_{k,i,h})\ll_{\cK,\kappa,\eta} \frac{1}{k^{(2\delta-1)\alpha(n-d-\kappa)}}.$$
%\begin{align*}
%\mu(\cC_{k,i,h})&\lesssim_{n,\cK}\frac{|\overline{A}_{k,i,h}|}{\left(|\underline{A}_{k,i,h}|^{\delta}-(|\overline{A}_{k,i,h}|-|\underline{A}_{k,i,h}|)\right)^2}\\
%&\asymp \frac{|\overline{A}_{k,i,h}|}{|\underline{A}_{k,h}|^{2\delta}}\asymp \frac{1}{k^{(2\delta-1)\alpha(n-2-\kappa)}}.
%\end{align*}
Now, using \eqref{equ:relations} and the above estimate, and recalling $\#\cI_k=O_{\cK}(k^{d_n})$ and $M(k)\ll k^{\alpha(\eta+\beta)}$, we can estimate for $k>k_0$
\begin{align*}
a_k&\leq \mu(\bigcup_{h\in \cI_k}\bigcup_{0\leq i<M(k)}\cC_{k,i,h})\leq \sum_{h\in \cI_k}\sum_{i=0}^{M(k)-1}\mu(\cC_{k,i,h})\\
%&\lesssim_{\cK,n} \sum_{h\in \cI_k}\sum_{i=0}^{M(k)-1}\frac{|\overline{A}_{k,i,h}|}{\left(|\underline{A}_{k,i,h}|^{\delta}-(|\overline{A}_{k,i,h}|-|\underline{A}_{k,i,h}|)\right)^2}\\
&\ll_{\cK,\kappa,\eta}\sum_{h\in \cI_k}\sum_{i=0}^{M(k)-1}\frac{1}{k^{(2\delta-1)\alpha(n-d-\kappa)}}\ll_{\cK}% \frac{k^{d_n+\alpha(\eta+\beta)}}{k^{(2\delta-1)\alpha(n-d-\kappa)}}\\
\frac{1}{k^{(2\delta-1)\alpha(n-d-\kappa)-\alpha(\eta+\kappa)-d_n-1}},
\end{align*}
where we used that $\alpha\beta=\alpha\kappa+1$. As in the proof of Theorem \ref{thm:latticecounting} for the exponent we can estimate
\begin{align*}
(2\delta-1)\alpha(n-d-\kappa)-\alpha(\eta+\kappa)-d_n-1&> \alpha(n-d-\kappa)-2-\alpha(\eta+\kappa)-d_n-1\\
&=\alpha(n-d-2\kappa-\eta)-d_n-3>1,
\end{align*}
where for the last inequality we used that $\alpha>\frac{d_n+4}{n-d-2\kappa-\eta}$. %Now for any value of $\alpha$ satisfying 
%\begin{equation}\label{e:alpha2}
%\alpha>\frac{d_n+2}{(2\delta-1)(n-d-\kappa)-(\eta+\kappa)}
%\end{equation}
% we have that $\sum_k a_k<\infty$ as needed
%
%In order to insure we can find such a parameter $\alpha$ we need to assume that 
%
%$$
%1-\frac{n-d-\eta}{(d_n+4)(n-d-\kappa)}=\frac{(d_n+3)(n-d-\kappa))+(\eta+\kappa)}{(d_n+4)(n-d-\kappa)}<\delta
%$$
%
% 
%
%
%Finally, we note that $\delta\in (1-\frac{1}{\alpha(n-2-\kappa)},1)$ satisfies $(2\delta-1)\alpha(n-2-\kappa)>\alpha(n-2-\kappa)-2$. Thus for $\alpha>\frac{d_n+4}{n-2-2\kappa-\eta}$ we have the exponent
%$$(2\delta-1)\alpha(n-2-\kappa)-\alpha(\eta+\kappa)-d_n-1>\alpha(n-2-2\kappa-\eta)-d_n-3>1.$$
Hence the series  $\sum_k a_k$ is summable and this finishes the proof.
\end{proof}

\begin{proof}[Proof of Theorem \ref{thm:sim2}]
Let $\kappa\in (0,\frac{n-d-\eta}{2})$ and take another $\kappa'\in (\kappa, \frac{n-d-\eta}{2})$. By Theorem \ref{thm:simuapprox} there exists a full measure subset $\cE=\cE_{N(t),\kappa'}\subset G$ and some constant $\delta\in (0,1)$ such that for any $F=g\cdot F_0$ with $g\in \cE$, there exists some $t_F'>0$ such that for any $t>t_F'$ and any interval $I'\subset [-N(t), N(t)]$ with $|I'|=t^{-\kappa'}$ we have
\begin{equation}\label{equ:interme}
\left|\cN_F(I',t)-c_F|I'|t^{n-d}\right|< t^{\delta(n-d-\kappa')}=|I'|^{\delta}t^{\delta(n-d)}.
\end{equation}
Let $\nu=\frac12\min\{\kappa'-\kappa, (1-\delta)(n-d-\kappa')\}$ and for any $F$ as above, let $t_F=\max\{t_F', (2+c_F)^{\frac{1}{\nu}}\}$. %We will prove the theorem by showing that all the quadratic forms in the full measure subset $\cE$ satisfy \eqref{equ:allint}. 
Now, for this $F$ and any $t>t_F$, we first assume that $I\subset [-N(t),N(t)]$ with $|I|\geq t^{-\kappa}$ is of length a multiple of $t^{-\kappa'}$. That is, $M_I:=\frac{|I|}{t^{-\kappa'}}$ is a positive integer, and we have a partition $I=\bigsqcup_{i=1}^MI_i$ with each subinterval $I_i\subset I$ and $|I_i|=t^{-\kappa'}$. Applying \eqref{equ:interme} to each $I_i$ we get 
\begin{align*}
&\left|\cN_F(I, t)-c_F|I|t^{n-d}\right|=\left|\sum_{i=1}^{M_I}\left(\cN_F(I_i,t)-c_F|I_i|t^{n-d}\right)\right|\\
&\leq \sum_{i=1}^{M_I}\left|\cN_F(I_i,t)-c_F|I_i|t^{n-d}\right|<\sum_{i=1}^{M_I}|I_i|^{\delta}t^{\delta(n-d)}\\
&=M_It^{-\delta\kappa'}t^{\delta(n-d)}=\frac{|I|}{t^{-\kappa'}}t^{-\delta\kappa'}t^{\delta(n-d)}=|I|t^{(1-\delta)\kappa'+\delta(n-d)}\leq |I|t^{n-d-2\nu},
\end{align*}
where for the last inequality we used that $2\nu\leq (1-\delta)(n-d-\kappa')$. Now we consider the general case of an interval $I\subset [-N(t),N(t)]$ with $|I|\geq t^{-\kappa}$. There exist intervals $\underline{I}\subset I\subset \overline{I}$ %in $(-N(t),N(t))$ 
such that lengths of $\underline{I}$ and $\overline{I}$ are of multiples of $t^{-\kappa'}$ and $|\overline{I}|-|\underline{I}|=t^{-\kappa'}$. Since $\underline{I}\subset I\subset \overline{I}$ we have $\cN_F(\underline{I},t)\leq \cN_F(I,t)\leq \cN_F(\overline{I},t)$. This implies that for $t>t_F$
\begin{align*}
\left|\cN_F(I,t)-c_F|I|t^{n-d}\right|&\leq \max\left\{\left|\cN_F(\underline{I},t)-c_F|I|t^{n-d}\right|,\left|\cN_F(\overline{I},t)-c_F|I|t^{n-d}\right|\right\}\\
&\leq \max\left\{\left|\cN_F(\underline{I},t)-c_F|\underline{I}|t^{n-d}\right|,\left|\cN_F(\overline{I},t)-c_F|\overline{I}|t^{n-d}\right|\right\}+c_Ft^{n-d-\kappa'},
%&\leq |\overline{I}|t^{n-2-2\nu}+c_Q|I|t^{n-2-\kappa'+\kappa}<(2+c_Q)t^{-\nu}|I|t^{n-2-\nu}\leq |I|t^{n-2-\nu}.
\end{align*}
where for the second inequality we used triangle inequality and the bound 
$$\max\{|I|-|\underline{I}|,|\overline{I}|-|I|\}\leq |\overline{I}|-|\underline{I}|=t^{-\kappa'}.$$
Now since both $\underline{I}$ and $\overline{I}$ are of lengths multiples of $t^{-\kappa'}$, applying the above estimate for $\left|\cN_F(\underline{I},t)-c_F|\underline{I}|t^{n-d}\right|$ and $\left|\cN_F(\overline{I},t)-c_F|\overline{I}|t^{n-d}\right|$ and noting that $|I|\geq t^{-\kappa}$ we have
\begin{align*}
\left|\cN_F(I,t)-c_F|I|t^{n-d}\right|&\leq |\overline{I}|t^{n-d-2\nu}+c_F|I|t^{n-d-\kappa'+\kappa}\leq 2|I|t^{n-d-2\nu}+c_F|I|t^{n-d-2\nu}\\
&=(2+c_F)|I|t^{n-d-2\nu}< |I|t^{n-d-\nu},
\end{align*}
where for the second inequality we used the estimates $|\overline{I}|<2|I|$ and $2\nu\leq \kappa'-\kappa$, and for the last inequality we used that $t>t_F\geq (2+c_F)^{\frac{1}{\nu}}$. This completes the proof.
\end{proof}

\end{document}